\documentclass[english]{amsart}

\usepackage{hyperref}

\usepackage[english]{babel}
\usepackage{amssymb}
\usepackage{amsthm}
\usepackage{mathrsfs}
\usepackage{amsmath}
\usepackage{amsfonts}
\usepackage{graphicx}
\usepackage{subcaption}
\graphicspath{ {C:\Users\aguilaro\OneDrive\Documentos\Papers\Presentation general LAC surface\Boundary manifold methods} }
\usepackage{tikz-cd}
\usepackage{tikz}
\usetikzlibrary{decorations.pathreplacing}
\usetikzlibrary{tikzmark}
\usepackage{xcolor}

\usetikzlibrary{calc} 
\usetikzlibrary{arrows.meta,bending,decorations.markings,intersections} 

\theoremstyle{plain} 

\newtheorem{thm}{Theorem}[section]
\newtheorem{lem}[thm]{Lemma}
\newtheorem{prop}[thm]{Proposition}
\newtheorem{cor}[thm]{Corollary}

\newtheorem{ques}{Question}

\theoremstyle{definition}
\newtheorem{defn}{Definition}[section]
\newtheorem{exmp}{Example}[section]

\theoremstyle{remark}
\newtheorem{rem}[thm]{Remark}

\newtheoremstyle{case}{}{}{}{}{}{:}{ }{}
\theoremstyle{case}

\DeclareMathOperator{\Ima}{Im}

\DeclareMathOperator{\Hom}{Hom}
\DeclareMathOperator{\End}{End}

\DeclareMathOperator{\codim}{codim}
\DeclareMathOperator{\Def}{Def}
\DeclareMathOperator{\Gr}{Gr}
\DeclareMathOperator{\Pic}{Pic}

\DeclareMathOperator{\Ext}{Ext}

\DeclareMathOperator{\MHS}{MHS}
\DeclareMathOperator{\AJ}{AJ}
\DeclareMathOperator{\Sym}{Sym}
\DeclareMathOperator{\Var}{Var}
\DeclareMathOperator{\Hdg}{Hdg}
\DeclareMathOperator{\coker}{coker}

\DeclareMathOperator{\NEW}{R}
\DeclareMathOperator{\Span}{span}
\DeclareMathOperator{\Alb}{Alb}
\DeclareMathOperator{\Id}{Id}
\DeclareMathOperator{\GL}{GL}

\newcommand{\abs}[1]{\left\vert#1\right\vert}

\newcommand{\Z}{\mathbb{Z}}
\newcommand{\C}{\mathbb{C}}
\newcommand{\Pp}{\mathbb{P}}

\newcommand{\Oo}{\mathscr{O}}
\newcommand{\Fg}{\mathcal{F}_g^R}
\newcommand{\Kg}{\mathcal{K}_g^R}
\newcommand{\Mg}{\mathcal{M}_g}

\begin{document}

\title[]{Infinitesimal invariants of mixed Hodge structures}
 \author{Rodolfo Aguilar, Mark Green and Phillip Griffiths}
      \begin{abstract} We introduce the notion of infinitesimal variations of mixed Hodge structures and invariants associated to them. We describe these invariants in the case of a pair $(X,Y)$ with $X$ a Fano 3-fold and $Y$ a smooth anticanonical K3 surface and in more detail in the case when $X$ is a cubic threefold. In this last setting, we obtain a generic global Torelli theorem for pairs.
      \end{abstract}
      
      
\address{Department of Mathematics, University of Miami, Coral Gables, FL 33146 \hfill\break\indent
{\it E-mail address}\/: {\rm aaguilar.rodolfo@gmail.com}\hfill\break\indent
\hfill\break\indent
Department of Mathematics, University of California at Los Angeles,\hfill\break\indent Los Angeles, CA 90095\hfill\break\indent
 {\it E-mail address}\/: {\rm mlg@ipam.ucla.edu}\hfill\break\indent 
\hfill\break\indent 
  \parbox{\linewidth}{Department of Mathematics, University of Miami, Coral Gables, FL 33146, and \hfill\break\indent Institute for Advanced Study, Einstein Drive, Princeton, NJ 08540\hfill\break\indent
{\it E-mail address}\/: {\rm pg@ias.edu}}   
}

      
\maketitle

\tableofcontents
\section{Introduction}
Infinitesimal variations of Hodge structures, or IVHS for short, has been used for some time to provide a tool to connect variational Hodge theory and geometry. They were introduced and parts of the theory developped in \cite{CGGH83,GH83,G83}. For additional results and further applications see \cite{PS82}, \cite{G94}, \cite[Partie VI]{V02}, \cite[Chapter 5]{CMSP17}.

In practice it provides connection 

\[\begin{tikzcd}
\overset{\text{geometry}}\bullet \ar[dash]{rr} \ar[dash]{rd}& & \overset{\text{Hodge theory}}\bullet\ar[dash]{ld} \\
 & \overset{algebra}\bullet & 
\end{tikzcd}
 \]
 
\noindent As such it is complementary to the use of Hodge theory in the cohomology of algebraic varieties. Since IVHS uses only the complex vector space $V$ and not the $\mathbb{Q}$-structure it may seem a bit surprising that it has proved as useful as it has for geometric applications. Since it does not involve the $\mathbb{Q}$-structure of $V$ it is not expected to be able to lead to existence results. However it can be used for generic global uniqueness results such as that of Donagi-Green \cite{DG84} and Theorem \ref{thm:Intro2} below.

Here, our first goal is to define the general notion of infinitesimal variations of mixed Hodge structure, or IVMHS for short. We are particularly interested in it as a method to study the geometry of smooth pairs $(X,Y)$, that is, both $X$ and $Y$ are smooth and $Y\subset X$. In this case the mixed Hodge structure on $H^n(X\setminus Y)$ can be seen as a two-step extension of pure Hodge structures and the structure of the mixed period domain is particularly simple \cite{U84,C85}. We can define an infinitesimal invariant of this extension class by using the derivative of the period map; in spirit, this is very similar to the infinitesimal invariant of a normal function \cite{G89}. To identify the basic properties of this invariant can be seen as the second main goal of this note.

The third goal is to extensively illustrate the structure of the IVMHS in the following setting: we specialize to the case when $X$ is a smooth projective $n$-dimensional Fano and $Y$ is a smooth  anti-canonical section. We obtain a $n$-symmetric form out of this pair and we prove an infinitesimal Torelli theorem for these anti-canonical pairs. This is done in section \ref{sec:IVMHS}. Note that an infinitesimal Torelli theorem for anti-canonical pairs $(X,D)$ with $X$ a surface and $D$ a SNC-divisor is proved in \cite[Thm 3.16]{F16}.

When $X$ is a Fano threefold and $Y$ an anticanonical K3 surface, the $3$-symmetric form can be seen as an analog of the Yukawa coupling in the log Calabi-Yau case. This form appeared in work of Markusevich, see \cite{M08}.  We can then relate the infinitesimal invariant of the extension class of $H^3(X\setminus Y)$ to the first derivative of this cubic form. Moreover, we can relate this infinitesimal invariant with an Abel-Jacobi type map. This is done in Section \ref{sec:FK3}.

If we let $X$ be the cubic threefold, we can use the description of the Hodge filtration of hypersurfaces as described in \cite{CG80} or \cite{G69} to get more information about the cubic form and the infinitesimal invariant of the extension class. In particular, we obtain conditions for the triviality of the cubic form, for its non-degeneracy and smoothness. With the computer aid of Macaulay2, we can prove the following theorem. 

\begin{thm} For $X$ a cubic threefold and $Y$ an anticanonical smooth section, both of them general, the cubic $C$ is smooth.
\end{thm}

Then we find that the dependence of the cubic form $C$ on the coefficients of the polynomial $F$ defining $X$ is of very high degree. After finding explicit formulas for the tangent space at several deformation spaces such as deformation of pairs $(X,Y)$, deformations in the normal direction of $Y$, we prove that the infinitesimal invariant for the MHS of $H^3(X\setminus Y)$ is non-zero. This implies that this MHS is non-split. 

We end section \ref{sec:CubThree} with a proof of a generic Torelli theorem for pairs.

\begin{thm}\label{thm:Intro2} Let $X$ be the cubic $3$-fold and $Y$ an anticanonical $K3$ surface. Assume that they are generic. Then the MHS on $H^3(X\setminus Y)$ determines the pair $(X,Y)$.
\end{thm}

In fact, our proof is construtive. It utilizes a second order invariant involving the cubic form. Thus, we can rephrase the theorem as saying that we can recover the pair $(X,Y)$ from the $2$-jet of the period maps for $X,Y, X\setminus Y$ for $X,Y$ generic. To the best of our knowledge, the only instance where some type of Torelli theorem for pairs was known, other than infinitesimally, was the surface anti-canonical pairs as in \cite{GHK15}, see also \cite{F16}.

There is a vast literature on the subject of Fano varieties, log-Calabi Yau pairs and the geometry of cubic hypersurfaces, see for example \cite{I99, H23} and references there-in. Here we have tried to illustrate the use of IVMHS in some geometric questions concerning these varieties.

\section{Infinitesimal variations of mixed Hodge structures}\label{sec:IVMHS}

\subsection{Definitions}

\begin{defn} A Hodge structure  $(V,F^\bullet)$ of weight $n$ consists of a finite-dimensional $\C$-vector space $V$ and an decreasing filtration $F^\bullet$ such that 
$$
\left\lbrace
\begin{array}{l}
F^n\subset F^{n-1}\subset \ldots \subset F^1\subset F^0=V\\
F^p\oplus \bar{F}^{n-p+1}\overset{\sim}\to V .
\end{array}\right.
$$

\end{defn}

A polarization $Q:V\times V\to \C$ is a bilinear form satisfying the Hodge bilinear relations.

By considering the set of filtrations $\{F'^\bullet\}$ of $V$ such that the $\dim F'^p=\dim F^p$ and that satisfy the Hodge bilinear relations with respect to $Q$, we obtain the so-called \emph{period domain} $D$.

When we have a variation of Hodge structures or VHS, see \cite{CMSP17}, over a quasi-projective smooth base $S$, there is associated a \emph{period map}
$$\Phi:S\to \Gamma\backslash D. $$

For $s\in S$, let $T:= T_s S$ be the tangent space at this point. We obtain a map $\Phi_s:T\to T_{\Phi(s)}D$ that gives a map
$$T\overset{\Phi}\to F^{-1}\End(V)/ F^0 \End(V). $$

\begin{defn} An Infinitesimal Variation of Hodge Structures or IVHS is a triple $(V, F^\bullet, T)$ where $T$ is a complex 
vector space, $V$ is a $\Sym T$-module, $F^\bullet$ is a decreasing filtration and $T\otimes F^p\to F^{p-1}$.
\end{defn}

\noindent We refer to \cite{G94} for a gentle introduction to IVHS, applications and relations with geometry.

\begin{rem}Usually $V$ is a $\mathbb{Q}$-vector space and we use $V_{\C}=V\otimes \C$. For our purposes here we can take $V$ to be a complex vector space.
\end{rem}
\begin{rem} Setting $H^{p,q}= F^p/F^{p+1}$ with $p+q=n$ equal the weight, we have that $\oplus^{p+q=n}H^{p,q}$ is a graded $\Sym T-$module. Frequently, IVHS passes to this instead of using $V$. This is similar to taking the Higgs bundle associated to a VHS. In principle, some information is lost.
\end{rem}

\begin{defn} A mixed Hodge structure is a triple $(V,W_\bullet, F^\bullet)$ with $W_\bullet$ an increasing filtration of $V$ and such that $(\Gr_k^W(V),F^\bullet\cap W_{k}/W_{k+1})$ is a Hodge structure of weight $k$.
\end{defn}

\begin{defn} An IVMHS is $(V,W_\bullet, F^\bullet, T)$ where $T$ is a complex vector space, $V$ is a $\Sym$ $T$-module, $W_\bullet$ is an increasing filtration, $F^\bullet$ a decreasing filtration and $T\otimes F^p\to F^{p-1}$, $T\otimes W_k\to W_k$ such that
\begin{equation}\label{eq:IVMHS}
(\Gr_k^W V, F^\bullet\cap W_k V/W_{k-1}, T) \text{ is an IVHS}.
\end{equation}
\end{defn}

For both IVHS and IVMHS we have Koszul-type groups. For IVMHS, the filtration $W_\bullet$ gives a spectral sequence whose $E_1$-term arises from (\ref{eq:IVMHS}). 

\subsection{Mixed Hodge structures of length two}
Let $B$ be a mixed Hodge structure of length two
\begin{equation}\label{eq:MHSlength2}
0\to A\to B\to C\to 0.
\end{equation}
We have that $A=W_n(B), C=\Gr_{n+1}^W(B)$. In this case, $B$ has a special structure for the period domains $D_1,D_2$ of the graded pieces $A,C$ in relation to the usual fibration 
\[\begin{tikzcd} D_1 \ar[r] & D \ar[d]\\
 & D_2 \end{tikzcd} \text{ given by }\begin{tikzcd} B\ar[d] \\ \Gr_{n+1}^W(B)=C, \end{tikzcd}\]
 there is a map 
 \[\begin{tikzcd}[contains/.style = {draw=none,"\in" {description},sloped}]
D\arrow[r]  & D_1\times D_2  \\
B \ar[r]\ar[u,contains] &  (W_n(B), \Gr_{n+1}^W(B)) \ar[u,contains].
\end{tikzcd}\]
If we have a $\mathbb{Z}$-structure then we have that the set of equivalences of MHS's $B$ in (\ref{eq:MHSlength2}) is 
\begin{equation}\label{eq:ExtTorus}
\Ext^1_{\MHS}(C,A)=\frac{\Hom_{\C}(C,A)}{F^0\Hom_{\C}(C,A)+\Hom_{\mathbb{Z}}(C,A)}.
\end{equation}
See for example \cite{U84,C85}.

This is a compact complex torus which contains a maximal abelian subvariety $E_a\subset \Ext_{\MHS}^1 (C,A)$ whose tangent space lies in the $(0,-1)\oplus (-1,0)$ part of the tangent space to (\ref{eq:ExtTorus}). If we have a VMHS with parameter space $S$ and where $A, C$ are constant the period mapping 
\begin{equation}\label{eq:PeriodMapping}
\Phi: S\to e+E_a
\end{equation}
maps to a translate of $E_a$.

If $S$ is a quasi-projective variety then the period mapping (\ref{eq:PeriodMapping}) extends to any smooth completion $\bar{S}$ and there is a factorization (omitting the $e$ in (\ref{eq:PeriodMapping}))
\begin{equation}
\begin{tikzcd}[contains/.style = {draw=none,"\subset" {description},sloped}]
S\arrow[r] \ar[d,contains]  & E_a  \\
\bar{S} \ar[r] &  \Alb(\bar{S})  \ar[u].
\end{tikzcd}
\end{equation}

Here we only treat the case when the weight of $C$ equals the weight of $A$ plus one. There are interesting cases when the weight of $C$ is bigger than the weight of $A$ plus one.

\begin{rem}\label{rem:CanDec} For a MHS $(V,W_\bullet, F^\bullet)$ we have the canonical Deligne splitting $V=\oplus V^{p,q}$ where $$W_k=\oplus_{p+q\geq k} V^{p,q}, \quad F^p=\oplus_{p'\geq p} V^{p,q} $$
and where $\bar{V}^{p,q}\cong V^{q,p} \mod W_{p+q-2}$. Thus length two MHS's have a canonical $V^{p,q}$ decomposition where $\bar{V}^{p,q}=V^{q,p}$.
\end{rem}

\subsection{Delta Invariant for extensions of Hodge Structures}\label{ss:DeltaInvariantGeneral}
Let $A$ and $B$ be pure Hodge structures. We have a MHS given as an extension by the exact sequence
$$0\to A\overset{f}\to B \overset{g}\to C\to 0. $$
Now, suppose we vary these over a base $S$. The derivative of the period map is
$$d\Phi_B\in T_S^*\otimes F^{-1}\End_\mathbb{C}(B)/F^0\End_\mathbb{C}(B). $$
Then, $d\Phi_B$ has the form $$\left( \begin{array}{ll}
d \Phi_A & \alpha \\
0 & d\Phi_C
\end{array}\right).$$
To compute $\alpha$, let us have 
\[\begin{tikzcd}
B \arrow[r,"g"]  
 & C \arrow[l, bend left=50, "\mu_\mathbb{Z}"] \arrow[l, bend right=50, "\mu_F" above]
\end{tikzcd},
\]
where $\mu_F\in F^0\Hom_\mathbb{C}(C,B), \mu_\mathbb{Z}\in \Hom_\mathbb{Z}(C,B)$. The extension class $e\in \Ext^1_{HS}(C,A)$ is defined by 
$$\mu_F-\mu_\mathbb{Z}=f\circ e .$$

If $\mu_F^\sharp, \mu_\mathbb{Z}^\sharp$ are two different choices of sections, then 
\begin{align*}
&\mu_F^\sharp = \mu_F+f\circ s_F, &&s_F\in F^0 \Hom_\mathbb{C}(C,A),\\
&\mu_\mathbb{Z}^\sharp=\mu_\mathbb{Z}+f\circ s_\mathbb{Z}, &&s_\mathbb{Z}\in F^0\Hom_\mathbb{Z}(C,A).
\end{align*}
Now,
$$F^p B=f\circ F^p A \oplus \mu_F(F^p C). $$
Since $f,g$ are defined over $\mathbb{Z}$, they are constant under $\nabla$. Hence,
$$\nabla_B F^pB=f\circ \nabla_A F^p A\oplus \left(\mu_F\circ \nabla_C F^pC +\nabla\mu_F(F^p C) \right). $$
Then,
$$\alpha=\nabla\mu_F \in T_S^*\otimes F^{-1}\Hom_\mathbb{C}(C,A).$$
The ambiguity that comes from the choice of $\mu_F$ is
$$\nabla(f\circ s_F)=f\circ \nabla s_F. $$
Now, as $\mu_F-\mu_\mathbb{Z}=f\circ e$, we have $\nabla\mu_F=f\circ \nabla e$ because $\nabla \mu_\mathbb{Z}=0$. Using $\mu_F$ to decompose $B=A\oplus \mu_F(C)$,
$$\alpha=\nabla e\in T_S^*\otimes F^{-1}\Hom_\mathbb{C}(C,A).$$
Taking account of the ambiguity,
$$\alpha=\delta e\in \frac{T_S^*\otimes F^{-1}\Hom_\mathbb{C}(C,A)}{\nabla F^0 \Hom_\mathbb{C}(C,A)}.$$
Using $\nabla^2=0$, we have that $\delta e$ sits in the cohomology at the middle of 
\begin{equation}\label{eq:SSExt}
F^0\Hom_\mathbb{C}(C,A)\overset{\nabla}\to T_S^*\otimes F^{-1}\Hom_\mathbb{C}(C,A)\overset{\nabla}\to \wedge^2 T_S^*\otimes F^{-2}\Hom_\mathbb{C}(C,A)
\end{equation}
with $\nabla=\nabla_{\Hom(C,A)}$.

\begin{defn} The infinitesimal invariant $\delta e$ of the extension class $e$ is the  class $\nabla e$ at the middle cohomology in (\ref{eq:SSExt}). 
\end{defn}


\subsection{Fano n-folds}\label{ss:FanoNForm}

In this subsection we let $X$ be a Fano $n$-fold. By taking a global section of $-(n-2)K_X$, we will construct a symmetric $n$-form in $H^1(\Omega_X^{n-1})^\ast$. It will be further studied in sections \ref{sec:FK3} and \ref{sec:CubThree} in the case when $n=3$. In this last case, the cubic form appeared in \cite[Prop. 3.1]{M08}, where Markusevich identifies it to a cubic form introduced by Donagi-Markman in their study of algebraically completely integrable Hamiltonian systems \cite{DM93}.

\begin{lem} Let $X$ be a Fano $n$-fold. Then there exists a map 
$$H^0(-(n-2)K_X)\to S^n H^1(\Omega_X^{n-1})^* $$
\end{lem}
\begin{proof}
Recall that for a $n$-dimensional vector space $U$, we have that $$ \bigwedge^{n-1} U  \cong (\det U) \otimes U^*,$$
hence
\begin{align*}
\bigwedge^n(\bigwedge^{n-1}U) & \cong (\det U)^n\otimes \bigwedge^n U^\ast \\
& \cong (\det U)^{n-1}.
\end{align*}
Globally, we have $\wedge^n \Omega_X^{n-1}\overset{\sim}\to (\Omega_X^n)^{\otimes(n-1)}$, so 
$$S^n H^1(\Omega^{n-1})\to H^n(K_X^{n-1})\cong H^0(-K_X^{n-2})^*, $$
which is dual to 
$$H^0(-(n-2)K_X)\to S^n H^1(\Omega_X^{n-1})^*.$$

Recall that, by horizontality, we have that the period map sends 
$$H^1(T_X)\to \Hom(H^{n-1,1}, H^{n-2,2}), $$
with $H^{i,j}:=H^{i,j}(X)$.
 
We have then that 
$$S^{n-2}H^1(T_X)\to \Hom(H^{n-1,1},H^{1,n-1})\cong S^2( H^{n-1,1})^*, $$
by the iterated period map. So, since 
$$H^0(-K_X)\to H^1(T_X)\otimes (H^{n-1,1}(X))^*, $$
we have that
\[\begin{tikzcd}S^{n-2} H^0(-K_X)\arrow[r]& S^{n-2}H^1(T_X)\otimes S^{n-2}(H^{n-1,1})^\ast \arrow[d] &  \\
& S^2 (H^{n-1,1})^*\otimes S^{n-2}(H^{n-1,1})^* \arrow[r] &  S^n(H^{n-1,1})^* ,
\end{tikzcd}
 \]
 and this factors through $H^0(-(n-2)K_X)$.
\end{proof}

\subsection{Infinitesimal Torelli for anti-canonical pairs}
The content of this subsection is well-known but we use it to fix some notation. The case when $X$ is of dimension two and $Y$ is a SNC-divisor is treated in \cite[Thm 3.16]{F16}.

Let $(X,Y)$ be a smooth anti-canonical pair, that is, $X$ and $Y$ are smooth projective varieties and $Y\in \abs{-K_X}$. If $\dim X=n$, we choose 
\begin{equation}\label{eq:LogVolume}
\Omega\in H^0(\Omega^n(\log Y)),
\end{equation}
to have a fixed identification
\begin{equation}\label{eq:IdentificationCY}
K_X\otimes [Y]\cong \Oo_X.
\end{equation}

The classical proof of infinitesimal Torelli for smooth Calabi-Yau manifolds extends directly for the map 
\begin{equation}\label{eq:PeriodMap}
\Def(X,Y)\to D_{(X,Y)}
\end{equation}
where $\Def(X,Y)$ denotes the moduli space of deformations of the pair $(X,Y)$ and $D_{(X,Y)}$ denotes the period domain for the MHS on $H^n(X\setminus Y)$.

In the case at hand, points of $D_{(X,Y)}$ are $(V,W_\bullet,F^\bullet)$ where the weight filtration is 
$$W_n \subset W_{n+1}=V $$
and the Hodge filtration is 
$$F^n\subset F^{n-1}\subset \ldots \subset F^0=V_{\mathbb{C}}. $$
The first piece of (\ref{eq:PeriodMap}) is given by
\begin{equation}\label{eq:FirstPeriod}
(X,Y)\to F^n H^n(X\setminus Y)\subset F^{n-1}(X\setminus Y) 
\end{equation}
\begin{prop}\label{prop:CYInfTor} The differential of (\ref{eq:FirstPeriod}) is injective.
\end{prop}
\begin{proof}
Using $H^n(X\setminus Y, \mathbb{C})\cong \mathbb{H}^n(\Omega_X^\bullet (\log Y))$ and the degeneration of the Hodge-de Rham spectral sequence gives
$$F^n H^n(X\setminus Y)\cong H^0(\Omega_X^n(\log Y))\cong \C\cdot \Omega $$
$$\Gr_F^{n-1} H^n(X\setminus Y)\cong H^1(\Omega_X^{n-1}(\log Y)) $$
The differential of (\ref{eq:FirstPeriod}) is induced by 
\[\begin{tikzcd} H^1(T_X(-\log Y))\ar[r] \ar[d, "\cong "] & \Hom(H^0(\Omega_X^n(\log Y)), H^1(\Omega_X^{n-1}(\log Y))) \ar[d, "\cong"] \\
H^1(\Omega_X^{n-1}(\log Y)) \ar[r, dashed] & H^1(\Omega_X^{n-1}(\log Y))
\end{tikzcd}
\] where the identifications are given by (\ref{eq:IdentificationCY}). The dashed arrow is then the identity.
\end{proof}

For a diagram in the case of dimension 3, see (\ref{eq:DifDiag}).

\section{Fano-K3 pairs}\label{sec:FK3}
\subsection{Fano Threefolds and anti-canonical sections}
\subsubsection{Fano threefolds}
In this section, we will consider the MHS's on $H^3(X-Y)$ and dually on $H^3(X,Y)$, where $X$ is smooth Fano $3$-fold and $Y\in \abs{-K_X}$ is a smooth anti-canonical divisor. We set 
\begin{align*}
2g-2=(-K_X)^3 \text{ with $g$ the genus of } X\\
R=\Ima\{H^2(X,\mathbb{Z})\to H^2(Y,\mathbb{Z})\}\subset \Pic (Y)
\end{align*}
and use the notation
\begin{align*}
&\Fg=  \left\lbrace \begin{array}{l}
 \text{moduli stack of pairs $(X,Y)$, with $X$ a Fano threefold of genus $g$ and}\\
\text{$Y$ a smooth surface in $\abs{-K_X}$ with a lattice isomorphism $R\cong \Pic(Y)$}\end{array}\right.\\
&\Kg= \left\lbrace\begin{array}{l}
\text{moduli space of smooth polarized $K3$ surfaces of degree $2g-2$ with }\\
\text{a marking }R\subset \Pic(Y)
\end{array} \right.\\
&\Mg=\left\lbrace \begin{array}{l}
\text{moduli space of triples $(X,Y,f)$ where } X\in \mathcal{F}_g^R, Y\in \mathcal{K}_g^R \text{ and $f$ corresponds}\\
\text{to the graph $G_f\subset Y\times X$ of an embedding $y\mapsto (y,f(y)), y\in Y$.}
\end{array}\right. \end{align*} 

\subsubsection{Properties of $(X,Y)$}
We will use the following notation for the basic numerical invariants of the pair $(X,Y)$:
\begin{equation}\label{eq:NumIn}
\left\lbrace
\begin{array}{ll}
(i) & h^{2,1}(X)=\dim H^1(\Omega_X^2)=\dim J(X),\\
&  \text{where $J(X)$ is the intermediate Jacobian of $X$;}\\
(ii) & r=\dim{\Ima(H^2(X)\to H^2(Y))};\\
& \text{this map is injective by the weak Lefschetz theorem;}\\
(iii) & \text{the \emph{genus} $g$ of $X$ is given as above by}\\
& 2g-2=(-K_X)^3=(-K_X|_Y)^2;
\end{array} \right.
\end{equation}
Since $Y\in\abs{-K_X}$ and $X$ is Fano,
$$N_{Y/X}=K_X^{-1}|_Y $$
is an ample line bundle over a $K3$ surface. Thus
\begin{equation}\label{eq:VanNor}
h^i(N_{Y/X})=0, \ i>0
\end{equation}
and by Riemann-Roch
\begin{equation}\label{eq:RR}
h^0(N_{Y/X})=\frac{1}{2}N_{Y/X}^2+\chi(\mathscr{O}_Y)=g+1.
\end{equation}
Since $-K_X\to X$ is ample
\begin{equation}\label{eq:AN}
h^i(T_X)=h^i(T_X\otimes K_X\otimes K_X^{-1})=h^i(\Omega_X^2\otimes K_X^{-1})=0, \ i\geq 2
\end{equation}
where the last step is Akizuki-Nakano vanishing.

\noindent Since $X$ is Fano
\begin{equation}\label{eq:Fano}
h^0(\Omega_X^p)=0, \text{ for } p>0.
\end{equation}
Finally, we will make the assumption
\begin{equation}\label{eq:Tan}
h^0(T_X)=0.
\end{equation}
\subsubsection{Basic exact sequences and interpretation}
We will frequently use the following basic exact sequences
\begin{equation}\label{eq:BES1}
0\to T_X\otimes [Y]^{-1}\to T_X\to T_{X}|_Y\to 0.
\end{equation}
Note that $T_X\otimes [Y]^{-1}\cong \Omega_X^2$, where $[Y]$ denotes the line bundle associated to $Y$.
\begin{equation}\label{eq:BES2}
0\to T_Y\to T_X|_Y\to N_{Y/X}\to 0
\end{equation}
\begin{equation}\label{eq:BES3}
0\to T_X(-\log Y)\to T_X\to N_{Y/X}\to 0
\end{equation}
\begin{equation}\label{eq:BES4}
0\to T_X\otimes [Y]^{-1}\to T_X(-\log Y)\to T_Y\to 0
\end{equation}
combining the cohomology of (\ref{eq:BES3}), (\ref{eq:BES4}) and using the assumption (\ref{eq:Tan}) together with (\ref{eq:VanNor}), (\ref{eq:AN}), (\ref{eq:Fano}) gives the following commutative diagram:
\medskip

\begin{equation}\label{eq:ComDiam}
\begin{tikzcd}
& & 0 \arrow[d] & & & \\
& & \underset{{\color{red}g+1}}{ H^0(N_{Y/X})} \arrow[d] & & &\\
0\arrow[r]& \underset{{\color{red}h^{2,1}(X)}}{H^1(\Omega_X^2)}\arrow[r] & \underset{{\color{red}20-r+h^{2,1}(X)}}{H^1(T_X(-\log Y))}  \arrow[r, "d\pi_Y"]\arrow[d," d\pi_X"] &\underset{{\color{red}20}} {H^1(T_Y)} \arrow[r] & \underset{{\color{red}r}} {H^2(\Omega_X^2)} \arrow[r] &0 \\
& & \underset{{\color{red}20-r+h^{2,1}(X)-g-1}}{H^1(T_X)\arrow[d]} & & & \\
& & 0 & & & \\
\end{tikzcd}
\end{equation}
The dimensions have been written in red below the relevant group. The $d\pi_Y$ and $d\pi_X$ are the differentials of the mappings $\pi_Y, \pi_X$ in 
\begin{equation}\label{eq:DefSq1}
\begin{tikzcd} & \Def(X,Y)\arrow[ld,"\pi_X"']\arrow[rd,"\pi_Y "] & \\
\Def(X) & & \Def(Y)
\end{tikzcd} 
\end{equation}

Following \cite{B04}, we may refine the diagram by recalling that $R=\Ima \Pic(X)\to \Pic(Y)$. Then (\ref{eq:DefSq1}) becomes

\begin{equation}\label{eq:DefSq2}
\begin{tikzcd} & \Fg\arrow[ld,"\pi_X"']\arrow[rd,"\pi_{Y,R} "] & \\
\Def(X) & & \Kg
\end{tikzcd} 
\end{equation}
where $\pi_{X,R}$ is now a submersion onto an open set in a period domain which is Hermitian symmetric of type $IV$.

\subsection{Cubic form}\label{ss:CubicForms}
\subsubsection{Cubic form for Fano threefolds}\label{sss:CubicFano}  Recall that from \ref{ss:FanoNForm}, we have a cubic form given by as follows.
Fix a smooth projective Fano threefold $X$. We have a map:
$$H^0(-K_X)\otimes H^1(\Omega_X^2)\to H^1(T_X),$$
equivalently:
$$H^0(-K_X)\to H^1(T_X)\otimes H^1(\Omega_X^2)^*.$$
Now, the period map takes 
$$H^1(T_X)\to S^2 H^1(\Omega_X^2)^*,$$
taken together, we have a map 
\[\begin{tikzcd}
H^0(-K_X) \arrow{r} \arrow{rd} & S^2 H^1(\Omega_X^2)^*\otimes \arrow[d] H^1(\Omega_X^2)^*\\
& S^3 H^1(\Omega_X^2)^*
\end{tikzcd}\]
Thus, we have a map from $H^0(-K_X)$ to cubic forms on $H^1(\Omega_X^2)^*$. 

\subsubsection{Cubic complex}
\begin{prop}\label{prop:CubicComplex} There is a complex
$$H^0(-K_X)\to (H^{2,1}(X))^*\otimes H^1(T_X)\to \wedge^2 (H^{2,1}(X))^*\otimes H^{1,2}(X) $$
\end{prop}
\begin{proof}
Let $\tau\in H^0(-K_X)$ and $\omega_i$ be a basis for $H^{2,1}(X)$, 
$$\tau \mapsto \sum \omega_i^*\otimes (\tau \rfloor \omega_i)\mapsto \sum (\omega_i^*\wedge \omega_j^*)(\tau \rfloor \omega_i)\rfloor \omega_j. $$
Now on the form level, the contraction commutes, but on the cohomology level, it anti-commutes, so this is a complex.
\end{proof}

\subsubsection{Normal bundle to the period map fixing Y} In this subsection, we will identify the normal bundle to the period map while keeping $Y$ fixed. 

Let $Y\in \abs{-K_X}$ given by a section $s_Y\in H^0(X,-K_X)$ and denote its associated cubic form by $C$.

Using the short exact sequence (\ref{eq:BES4}), let us denote
$$H^1(T_Y)_R:=\ker H^1(T_Y)\to H^2(\Omega_X^2). $$

\begin{prop}\label{prop:NormalBundle} We have a natural map $$H^1(T_Y)_{\NEW}\to \frac{S^2 H^{2,1}(X)^*}{J_2(C)}.$$
\end{prop}
Here $J_2(C)$ denotes the subspace of $S^2 H^{2,1}(X)^*$ generated by the first partials of $C$.
\begin{proof}
Suppose that $Y$ is associated to a section $s_Y\in H^0(X,-K_X)$ and denote by $C$ its cubic form. We have that 
\[\begin{tikzcd}
H^{2,1}(X)\ar[r,"s_Y"]\ar[rd, "C"] & H^1(T_X) \ar[d, "d\Phi_X"]\\
 & S^2 H^{2,1}(X)^*
\end{tikzcd}
\]
where $d\Phi_X$ denotes the derivative of the period map of $X$.
The diagonal map is given by first partials of $C$.

Using again (\ref{eq:BES4})
\[\begin{tikzcd}
0\ar[r]&  H^1(\Omega_X^1) \ar[r]\ar[drr,bend right, "C" below] & H^1(T_X(-\log Y)) \ar[r] \ar[d]& H^1(T_Y)_R \ar[r]  & 0 \\ 
& & H^1(T_X) \ar[r] & S^2 H^1(\Omega_X^2)^* & 
\end{tikzcd} 
\]
so $H^1(T_Y)_R\to \frac{S^2 H^1(\Omega_X^2)^*}{J_2(C)}$.
\end{proof}

To interpret this, consider the following diagram
\begin{equation}\label{tkzdg:normalbundleperiod}
\begin{tikzcd} 
0 \ar[d] & \\
T_{S_Y}\ar[d]\ar[r] & S^2 H^{1,2}(X)\\
T_{\Fg}\ar[ur]\ar[d]& \\
H^1(T_Y)_{\NEW}\ar[d]\ar[r] & \frac{S^2 H^{1,2}(X)}{\Ima T_{S_Y}}\cong \frac{S^2 H^{1,2}(X)}{J_2(C)}\\
0 &
\end{tikzcd}
\end{equation}
and $\frac{S^2 H^{1,2}(X)}{\Ima T_{S_Y}}\cong \frac{S^2 H^{1,2}(X)}{J_2(C)}$ represents the normal bundle to the period map while keeping $Y$ fixed.

\subsubsection{Properties of cubics} Now, let us recall some general basic properties of cubics. 

\begin{lem}\label{lem:PropCubics} The following conditions are equivalent: 
\begin{itemize}
\item[-] The cubic $C$ is non-degenerate in $S^3 V^\ast$,
\item[-] there is no proper subspace $V_0\subset V$ such that $C\in S^3 V_0^\ast\subset S^3V^*$,
\item[-] the partials $\frac{\partial C}{\partial z_1^\ast}, \ldots, \frac{\partial C}{\partial z_5^\ast}$ are linearly independent in $S^2 V^\ast$.
\end{itemize}
\end{lem}

Recall that a cubic form $C$ is called non-degenerate if the corresponding cubic variety spans the ambient projective space.

There is also a similar statement for smoothness. 

\begin{lem}\label{lem:CubicAreSmooth} The following conditions are equivalent: 
\begin{itemize}
\item[-]The cubic $C$ is smooth, 
\item[-] $\Var(\frac{\partial C}{\partial z_1^\ast}, \ldots, \frac{\partial C}{\partial z_5^\ast})=\varnothing$ in $\mathbb{P}^4$, 
\item[-] $S^6 V^\ast / J_C=0$.
\end{itemize}
\end{lem}
Note that, for an abstract cubic, we have the following implications 
$$ C\not = 0 \Leftarrow C \text{ non- degenerate } \Leftarrow C \text{ smooth. } $$
For example, $z_1(z_1^2+\ldots+z_5^2)$ is non-degenerate, but not smooth.

\subsection{Delta Invariant for Fano-K3 pairs}\label{ss:DeltaInvariant} Let us compute now the invariant $\delta e$ of section \ref{ss:DeltaInvariantGeneral} for a Fano-K3 pair. We will see in Corollary \ref{cor:NormalBundle} that its non-trivial part can be related to the normal bundle to the period map while keeping $Y$ fixed.

Let $Y, X$ be as above. This is, smooth manifolds, with $X$ a Fano threefold and $Y\in \abs{-K_X}$ an anti-canonical section. Consider the exact sequence: 
\begin{equation}\label{eq:DeltaInvFano}
0\to \frac{H^2(Y)}{H^2(X)}\to H^3(X,Y)\to\ker(H^3(X)\to H^3(Y))\to 0.
\end{equation}

We have $H^3(Y)=0, H^2(X)=\Hdg^1(C), H^{3,0}(X)=0$. Following \cite{B04}, let us denote by $H^2(Y)_R=H^2(Y)/\Ima(H^2(X))$. The extension class of the MHS lies in 
$$e_{Y/X}\in \frac{\Hom_\mathbb{C}(H^3(X), H^2(Y)_R)}{F^0\Hom_\mathbb{C}(H^3(X),H^2(Y)_R)+\Hom_\mathbb{Z}(H^3(X),H^2(Y)_R)}. $$
The infinitesimal invariant $\delta e_{Y/X}:=\delta e$ from section \ref{ss:DeltaInvariantGeneral} lies in the middle cohomology of 
\[
\begin{tikzcd}
 {F^0\Hom_\mathbb{C}(H^3(X),H^2(Y)_R)} \arrow[r]
& \arrow[d, phantom, ""{coordinate, name=Z}] T_S^*\otimes F^{-1}\Hom_\mathbb{C}(H^3(X),H^2(Y)_R) \arrow[d,
"\nabla" above,
rounded corners,
to path={ -- ([xshift=2ex]\tikztostart.east)
|- (Z) [near end]\tikztonodes
-| ([xshift=-2ex]\tikztotarget.west)
-- (\tikztotarget)}] \\
 & \wedge^2 T_S^*\otimes F^{-2}\Hom_\mathbb{C}(H^3(X),H^2(Y)_R), 
\end{tikzcd}
\]
and hence
\[\begin{tikzcd}
\begin{array}{c}
H^{2,1}(X)^*\otimes H^{2,0}(Y) \ \ \ \\
\oplus \\
H^{1,2}(X)^*\otimes H^{1,1}(Y)_R
\end{array} \arrow[r, "\nabla"] &  \arrow[d, phantom, ""{coordinate, name=Z}]  \begin{array}{c}
T_S^*\otimes H^{2,1}(X)^*\otimes H^{1,1}(Y)_R\\
\oplus \\
T_S^*\otimes H^{1,2}(X)^*\otimes H^{0,2}(Y)
\end{array} \arrow[d,
"\nabla" above,
rounded corners,
to path={ -- ([xshift=2ex]\tikztostart.east)
|- (Z) [near end]\tikztonodes
-| ([xshift=-2ex]\tikztotarget.west)
-- (\tikztotarget)}] \\
 & \wedge^2 T_S^*\otimes H^{2,1}(X)^*\otimes H^{0,2}(Y).
\end{tikzcd} \]

Holding $Y$ constant, $\delta e_{Y/X}$ decomposes as 
$$\begin{array}{lc}
\delta {e}_a \in & H^{0,2}(Y)\otimes \ker(T_S^*\otimes H^{2,1}(X)\overset{\nabla_X}\to \wedge^2 T_S^*\otimes H^{1,2}(X))\\
\oplus & \oplus \\
\delta {e}_b \in & H^{1,1}(Y)_R\otimes \coker (H^{2,1}(X)\overset{\nabla_X}\to T_S^*\otimes H^{1,2}(X)).
\end{array} $$

Using (\ref{eq:BES4}), we obtain the following exact sequence
\begin{equation}\label{eq:TangentFiber}
\begin{tikzcd}[contains/.style = {draw=none,"\cong" {description},sloped}]
0=H^0(T_Y)\ar[r]& H^1(T_X(-[Y]))\arrow[r] \ar[d,contains] & H^1(T_X(-\log Y)) \ar[d,contains]\ar[r] & H^1(T_Y	), \ar[d,contains]\\
&H^1(\Omega_X^2) &  T_{\Fg} &   T_{\Kg}
\end{tikzcd}
\end{equation}
It follows then that $T_{S_Y}\cong H^1(\Omega_X^2)$, with $S_Y$ the fiber of $Y$ in $\Fg\to \Kg$.
 Now $S=S_Y$, then above formulas for $\delta e_a$ and $\delta e_b$ are
$$\begin{array}{lc}
\delta e_a \in & H^{0,2}(Y)\otimes \ker(H^{1,2}(X)\otimes H^{2,1}(X)\overset{\nabla_X}\to \wedge^2 H^{1,2}(X)\otimes H^{1,2}(X))\\
\oplus & \oplus \\
\delta e_b \in & H^{1,1}(Y)_R\otimes \coker (H^{2,1}(X)\overset{\nabla_X}\to H^{1,2}(X)\otimes H^{1,2}(X)).
\end{array} $$

\begin{prop}\label{prop:DeltaA} The map $\delta e_a$ is $$\Id \in {H^{2,1}}^*(X)\otimes H^{1,2}(X)\otimes H^{0,2}(Y).$$
\end{prop}
\begin{proof}
Recall that from (\ref{eq:TangentFiber}) that $T_{S_Y}=H^1(T_X(-[Y]))$. Consider now the exact sequence, 
$$0\to \Oo_X(-Y)\to \Oo_X \to \Oo_Y\to 0.$$
From this, as $H^{0,2}(X)=0$, then
$$H^{0,2}(Y)\cong \ker (H^3(\Oo_X(-Y))\to H^3(\Oo_X)=0).$$
The map $$T_{S_Y}\otimes H^{1,2}(X)\to H^{0,2}(Y),$$ i.e.
$$H^1(T_X(-Y))\otimes H^2(\Omega_X^1)\to H^3(\Oo_X(-Y)),$$ is cup product, using $T_X(-Y)\otimes \Omega_X^1\to \Oo_X(-Y)$.
Now, in our case $Y\in \abs{-K_X}$, this becomes
$$H^1(\Omega_X^2)\otimes H^2(\Omega_X^1)\to H^3(K_X)\cong \mathbb{C},$$
by duality.
So $T_{S_Y}\to H^{2,1}(X)\otimes H^{0,2}(Y)$ is $H^{2,1}(X)\overset{\Id}\to H^{2,1}(X)$, identifying $H^{0,2}(Y)\cong \mathbb{C}$. 
\end{proof}

We have therefore the following diagram: 
\[
\begin{tikzcd}
{\Id}\in H^{2,1}(X)\otimes H^{1,2}(X) \ar[rr]\ar[dr,"\text{partials of }C\otimes \Id "] \ar[dd, "C"] & & \wedge^2 H^{1,2}(X)\otimes H^{1,2}(X)\\
& S^2 H^{1,2}(X)\otimes H^{1,2}(X)\ar[ur] & \\
S^3 H^{1,2}(X)\ar[ur]& &
\end{tikzcd}
\]
where $C\in S^3 H^{1,2}(X)$ is the cubic form associated to $Y\in H^0(-K_X)$. 

From Proposition \ref{prop:NormalBundle}, we obtain the following description of $\delta e_b$. 
 
\begin{cor}\label{cor:NormalBundle} The map for $\delta e_b$ is $H^{2,1}(X)\to S^2H^{1,2}$ by first partials of $C$. Hence
$$\delta e_b\in H^{1,1}(Y)_R\otimes \frac{S^2 H^{1,2}(X)}{J_2(C)}. $$
\end{cor}
Note that from (\ref{tkzdg:normalbundleperiod}), we know that $\frac{S^2 H^{1,2}(X)}{J_2(C)}$ is the normal bundle of the image of $TS_Y$,  the tangent bundle of fibers of $\Fg\to \Kg$.

\subsection{Period map in dimension three}
In the case of consideration, $(X,Y)$ where $X,Y$ are smooth, $X$ is a Fano threefold, and $Y\in \abs{-K_X}$, it is well known that the deformation space of pairs $\Def(X,Y)$ is smooth with tangent space $H^1(T_X(-\log Y))$.

\medskip

\noindent The differential of the period mapping for the mixed Hodge structure on $H=H^3(X\setminus Y)$, obtained from the exact sequence
\begin{equation}\label{eq:ExSeqHdg1}
0\to H^3(X)\to H^3(X\setminus Y)\to H^2(Y)(-1)\to H^4(X)\to 0,
\end{equation}
 is:
\begin{equation}\label{eq:diff}
H^1(T_X(-\log Y))\to \Hom(\Gr_F^3 H,\Gr_F^2 H)\oplus \Hom(\Gr_F^2 H,\Gr_F^1 H).
\end{equation}
The proof of Proposition \ref{prop:CYInfTor} can be rewritten as follows.

 We have that
\begin{equation}\label{eq:Gr}
\begin{array}{l}
\Gr_F^3 H \cong H^0(\Omega_Y^2)(-1)=:H^{3,1},\\
\Gr_F^2 H \cong H^1(\Omega_Y^1)(-1)\oplus H^1(\Omega_X^2)=:H^{2,2}\oplus H^{2,1},\\
\Gr_F^1 H \cong H^2(\mathscr{O}_Y)(-1)\oplus H^2(\Omega_Y^1):=H^{1,3}\oplus H^{1,2}.
\end{array}
\end{equation}

Here, we have use the canonical splitting as in remark \ref{rem:CanDec}.

The differential of the period mapping may be pictured by 

\begin{equation}\label{eq:DifDiag}
\begin{tikzcd}
\overset{\ \ \ H^{3,1}}\bullet   \arrow[r,red, "(1)"]  \arrow[dr, ,blue, "(1)"] & \overset{\ \ \ H^{2,2}}\bullet \arrow[dr,blue,"(2)"] \arrow[r,red,"(1^*)"] & \overset{\ \ \ H^{1,3}}\bullet \hspace{15pt} \tikzmark{bracebegin}   \\
    & \underset{H^{2,1}}\bullet \arrow[r,red, "(2)"] &  \underset{H^{1,2}}\bullet \ \  \big\rbrace W_3  \hspace{2pt} \tikzmark{braceend} 
\end{tikzcd}
\end{equation}

\begin{tikzpicture}[overlay,remember picture]
  \draw[decorate,decoration={brace}] ( $ (pic cs:bracebegin) +(0, 10pt)  $ ) -- node[right] {$W_4$} ( $ (pic cs:braceend) -(0, 5pt) $ );
  \draw[gray, thick] (4.2,2.5) --node[left] {$F^3$} (4.2,.5);
  \draw[gray, thick] (6,2.5) --node[left] {$F^2$} (6,.5);

  \end{tikzpicture}

Here we represent the tangent space to a family in $\Def(X,Y)$ by a bullet point $\bullet$ and the arrows represent its action on the graded pieces. 
  
Using a choice of $\Omega\in H^0(\Omega_X^3(\log Y))\cong \mathbb{C}$, we have: 
\begin{align*}
H^1(T_X(-\log Y))\cong H^1(\Omega_X^2(\log Y))\\
\Gr_F^3 H\cong \mathbb{C}, \Gr_F^2 \cong H^1(\Omega_X^2(\log Y))
\end{align*}
and then $\color{red} (1)\color{black}\oplus \color{blue} (1)  $ equals the identity.

\begin{rem} \begin{enumerate}
\item  In (\ref{eq:Gr}) and in the diagram (\ref{eq:DifDiag}), we have split the weight filtration. Thus, e.g., the mapping $\color{blue} (1)$ is only intrinsically defined in the kernel of $\color{red} (1)$. In more detail, for a given $\eta \in H^1(T_X(-\log Y))$, there is a derivative of the period map 

\[\begin{tikzcd} & 0\ar[d] \\
& H^{2,1} \ar[d] \\
\C\cong H^{3,1}\cong F^3 H^3(X\setminus Y) \arrow[ur,"(1)" ,dashed, color=blue] \ar[r] \ar[rd, "(1)", color=red] & \frac{F^2 H^3(X\setminus Y)}{F^3 H^3(X\setminus Y)} \ar[d]\\
& H^{2,2}\ar[d] \\
& 0
\end{tikzcd}
\]
using the notation of (\ref{eq:DefSq1}), we have that: {\color{red} (1)} depends only on $\pi_Y(\eta)$ and therefore the kernel of {\color{red} (1)} on $T\Def(X,Y)$ equals the kernel of $\pi_Y$. On the other hand, {\color{blue} (1)} is only well-defined on $\ker \pi_Y
$, and is an isomorphism from $\ker \pi_Y\subset T\Def(X,Y)$ to $H^{2,1}$.

\item Thus, not only is infinitesimal Torelli true for pairs $(X,Y)$ as above, the differential of the period mapping is injective on the first piece in (\ref{eq:diff}) of that mapping. \end{enumerate}
\end{rem}

The top row in (\ref{eq:DifDiag}) is intrinsically defined; i.e., no splitting of filtrations is involved. Moreover, there is a natural mapping 
\begin{equation}\label{eq:maptoK3}
\Def(X,Y)\to \Def(Y)
\end{equation}
and the top row may be identified with the differential of the period mapping of the polarized $K3$ surface $Y$. Moreover, $\color{red} (1^*)$ is the dual of $\color{red} (1)$.
\begin{rem} By what was just said the kernel of the top row in (\ref{eq:DifDiag}) defines the tangent space to the fibres of (\ref{eq:maptoK3}). We may think of such a fiber as a deformation $X_t$ of $X$ where all $X_t$ contain a fixed $Y$ (an anchored deformation of $(X,Y)$). On such a fiber the differential of the period mapping of the $X_t$ is given by $\color{red}(2)$ in the bottom row of the diagram (\ref{eq:DifDiag}). In principle, then we have:

On the kernel of $\color{red}(1)$ the mapping $\color{red}(2)$ is determined by the mapping $\color{blue}(1)$.
\end{rem}

\begin{ques} How can we express $\color{red}(2)$ in terms of $\color{red}(1)$ $+$ $\color{blue}(1)$?
\end{ques}

We have a dual sequence to (\ref{eq:ExSeqHdg1}),
\begin{equation}\tag{\ref{eq:ExSeqHdg1}*}
0\to H^2(X)\to H^2(Y)\to H^3(X,Y)\to H^3(X)\to 0.
\end{equation}
This induces a MHS on $H^3(X,Y)$ which is dual to that of $H^3(X\setminus Y)$. Let us denote $H^{1,1}(Y)/H^{1,1}(X)$ by $H^{1,1}(Y)_{\NEW}$. The dual of (\ref{eq:DifDiag}) is:

\begin{equation}\tag{\ref{eq:DifDiag}*} \label{eq:DifDiag*}
\begin{tikzcd}
\overset{H^{2,1}(X)}\bullet \arrow[dr,blue,"\check{(2)}"] \arrow[r,red,"\check{(2)}"] & \overset{H^{1,2}(X)}\bullet \ar[dr,blue,"\check{(1)}"]   \\
 \overset{H^{2,0}(Y)}\bullet\arrow[r,red,"\check{(1)}"]   & \overset{H^{1,1}(Y)_{\NEW}}\bullet \arrow[r,red,"(\check{1}^*)"] &  \overset{H^{0,2}(Y)}\bullet
\end{tikzcd}
\end{equation}

\subsection{Relation with Abel-Jacobi map}
Consider $f\in \Mg$, this is, $f:Y\to X$ and $Y\subset X$ as above. Denote by $G_f\in Y\times X$ its graph. We have
\begin{align*}
[G_f]\in H^6(Y\times X) \cong & (H^0(Y)\otimes H^6(X))\oplus (H^2(Y)\otimes H^4(X))\\
&\oplus (H^4(Y)\otimes H^2(X))\\
 \cong & \Hom(H^0(X),H^0(Y))\oplus \Hom(H^2(X),H^2(Y))\\ 
 &\oplus \Hom(H^4(X),H^4(Y)).
\end{align*}
Now, consider $p_0\in Y$ and 
\begin{align*}
&d_f=p_0\times f(Y), \ [d_f]\in H^4(Y)\otimes H^2(X)\cong \Hom(H^4(X),H^4(Y))\\
&c_f=Y\times f(p_0), \ [c_f]\in H^0(Y)\otimes H^6(X)\cong \Hom(H^0(X),H^0(Y))
\end{align*}
Now $[G_f-d_f-c_f]\in H^2(Y)\otimes H^4(X)\cong \Hom(H^2(X), H^2(Y))$. Recall that $H^2(X)\overset{f^*}\to H^2(Y)_{\NEW}$ is $0$. Note that $H^2(Y)=f^*H^2(X)\oplus H^2(Y)_{\NEW}$ is a direct sum of HS.

Define $Z_f=G_f-d_f-c_f$. It projects to zero in $H^6(Y\times X) \mod f^*H^2(X)\otimes H^4(X)$. Thus 
$$\AJ(Z_f)_{\NEW}\in \frac{H^5(Y\times X)}{F^3 H^5(Y\times X)+H_{\Z}^5(Y\times X)} \mod \Ima(f^*H^2(X)\otimes H^3(X)) $$
is well-defined, because $[Z_f]_R=0$.
\begin{thm} We have that $$\delta e_{Y/X}= d \AJ(Z_f)_{\NEW}$$ after appropriate identification. 
\end{thm}
\begin{proof}
The identification is 
\begin{align*}
\Hom(H^3(X),H^2(Y)_{\NEW})&\cong H^3(X)\otimes H^2(Y)_{\NEW}(-3)\\
F^p\Hom(H^3(X),H^2(Y)_{\NEW})&\cong F^{p+3}(H^3(X)\otimes H^2(Y)_{\NEW})
\end{align*}
so 
\begin{equation}\label{eq:IdentificationAJ}
\frac{F^{-1}\Hom(H^3(X),H^2(Y)_{\NEW})}{F^0\Hom(H^3(X),H^2(Y)_{\NEW})}\cong \frac{F^2(H^3(X)\otimes H^2(Y)_{\NEW})}{F^3(H^3(X)\otimes H^2(Y)_{\NEW})}
\end{equation}
and then $\delta e_{Y/X}$ and $d \AJ(Z_f)_{\NEW}$ belong to a complex whose middle term is the tensor product of $T^*$ and either term of (\ref{eq:IdentificationAJ}). For example, for $\delta e_{Y/X}$ we have as in \ref{ss:DeltaInvariant}

\[
\begin{tikzcd}
 {F^0\Hom_\mathbb{C}(H^3(X),H^2(Y)_R)} \arrow[r]
& \arrow[d, phantom, ""{coordinate, name=Z}] T_S^*\otimes F^{-1}\Hom_\mathbb{C}(H^3(X),H^2(Y)_R) \arrow[d,
"\nabla" above,
rounded corners,
to path={ -- ([xshift=2ex]\tikztostart.east)
|- (Z) [near end]\tikztonodes
-| ([xshift=-2ex]\tikztotarget.west)
-- (\tikztotarget)}] \\
 & \wedge^2 T_S^*\otimes F^{-2}\Hom_\mathbb{C}(H^3(X),H^2(Y)_R). 
\end{tikzcd}
\]

Now for the tangent space to $\Mg$ at $f$ we have that
$$T_{\Mg}(f)\cong H^0(Y,N_{f(Y)/Y\times X})\cong H^0(Y, f^* T_X), $$
and that $d\AJ:H^0(Y,f^* T_X)\to F^2 H^3(Y\times X)/F^3(Y\times X)$. This map can be described as follows: let
$$\alpha\in H^{p,q}(X), \eta \in H^0(T_X|_Y).$$
Now $\alpha|_Y \lrcorner \eta$ gives an element of $H^{p-1,q}(Y)$.
We have that $$H^{p,q}(X)^*\otimes H^{p-1,q}(Y)\cong H^{3-p,3-q}(X)\otimes H^{p-1,q}(Y)\subset H^{2,3}(X\times Y)=\frac{F^2 H^5(Y\times X)}{F^3H^5(Y\times X)}.$$
This induces the description of the element $d\AJ(Z_f)_R$ by passing to the quotient.

Now for the extension class, we have that from (\ref{eq:BES1}) and (\ref{eq:BES4}), it follows that $H^0(T_X|_Y)\subset H^1(\Omega_X^2)$ consists of those elements that let fixed $X$ and $Y$ in moduli. It induces two maps $$ H^0(T_X|_Y)\to H^{0,2}(Y)\otimes H^{2,1}(X)\cong H^{2,1}(X)$$
and 
$$H^0(T_X|_Y)\to H^{1,1}(Y)_R\otimes H^{1,2}(X).$$
These correspond to $\delta e_a$ and $\delta e_b$ from Proposition \ref{prop:DeltaA} and Corollary \ref{cor:NormalBundle} respectively.
\end{proof}

We have also a classical Abel-Jacobi map associated to the pair $(X,Y)$ and two embeddings $f, f_0:Y\to X$. Let $\lambda\in \Pic(Y)$, the algebraic $1$-cycle $f(\lambda)-f_0(\lambda)$  is homologous to zero in $X$. Thus 
$$\AJ_X(f(\lambda)-f_0(\lambda))\in J(X) $$
is well-defined. 

\section{The Cubic 3-fold}\label{sec:CubThree}
\subsection{Notation}
Let $X\subset \mathbb{P}V^\ast \cong \mathbb{P}^4$ be a cubic threefold and fix equations: $X=\{F=0\}$, take $Y\in \abs{-K_X}$ with $Y=\{F=0, Q=0\}$ and suppose that both $X,Y$ are smooth. Let us denote by $J_F$ the ideal in $S^\bullet V=\Sym^\bullet V$ generated by the first partials of $F$ and denote $J_F\cap S^k V$ by $J_{F,k}$.

Note that 
$$F\in S^3 V, \ Q\in S^2 V,$$
$$H^{2,1}(X)=V, \ H^{1,2}=S^4 V/J_{F,4}, \ H^{3,3}(X)\cong S^5 V/J_{F,5} \cong \mathbb{C} $$
Note that in this last isomorphism with $\mathbb{C}$, a choice of an isomorphism $K_{\mathbb{P}V^*}\cong \mathscr{O}_{\mathbb{P}V^\ast}(-5)$ is made.

\noindent There is a map $H^0(-K_X)\cong S^2 V\to S^3(V^\ast)\cong S^3 H^{2,1}(X)^\ast$ given by multiplication 
$$S^2 V\otimes S^3V\to S^5V/J_F \cong \mathbb{C}. $$

Let $C$ be the cubic form constructed in Section \ref{ss:CubicForms}. 

\begin{lem} The map $H^0(-K_X)\cong S^2 V\to S^3(V^\ast)\cong S^3 H^{2,1}(X)^\ast$ is induced by the cubic form $C$.
\end{lem}
\begin{proof}
Recall that the cubic form is defined by 
$$H^1(\Omega_X^2)\otimes H^0(-K_X)\to H^1(T_X) $$
and the derivative of the period map 
$$H^1(T_X)\to S^2 H^1(\Omega_X^2)^*.$$
So $H^0(-K_X)\to H^1(\Omega_X^2)^*\otimes S^2 H^1(\Omega_X^2)^*$.

We have also a natural map 
$$\wedge^3 T_X\to -K_X.$$

Since $$\Omega_X^2\cong T_X \otimes K_X,$$
we get
$$ \wedge^3 \Omega_X^2 \to 2 K_X. $$

Now, 
$$ S^3 H^1(\Omega_X^2)\to H^3(\wedge^3 \Omega_X^2)\to H^3(2_K)\cong H^0(-K_X)^*, $$
so $S^3H^1(\Omega_X^2)\to H^3(K_X^2)\cong H^0(-K_X)^*$.

 Dually, $H^0(-K_X)\to S^3 H^1(\Omega_X^2)^*$. This is the cubic form.
\end{proof}
\noindent We have the pairing $\langle,\rangle: S^k V^\ast \otimes S^k V\to \mathbb{C}$.

\subsection{Cubic form}

Here we will use the standard notation for $I, J$ ideals of a ring $R$
$$(I:J):=\{r\in R \mid rJ\subseteq I\} $$
We call it \emph{the ideal quotient}.
\subsubsection{Condition for triviality}

Let $C$ be the cubic form constructed in Section \ref{ss:CubicForms} corresponding to $Q\in \abs{-K_X}$. 

\begin{lem}\label{lem:ZeroCubic} We have that $C=0$ if and only if $Q\in J_{F,2}$. 
\end{lem}
\begin{proof}
We have that 
$$\frac{S^2 V}{J_F}\otimes \frac{S^3 V}{J_F}\to \frac{S^5 V}{J_F}\cong \mathbb{C} $$
is a non-degenerate pairing, so 
$$\frac{S^2 V}{J_F}\to (\frac{S^3 V}{J_F})^*\subseteq S^3 V^* $$
is injective.
\end{proof}

It can happen that $C=0$. For example, take $F=\sum_{i=0}^4 z_i^3$ and $Q=\sum_{i=0}^4 z_i^2$. Note that $X=\{F=0\}$ and $Y=\{F=Q=0\}$ are smooth.

\subsubsection{Relation of C to Q and F}

\begin{lem}\label{lem:MaximalIdeal} Let $m=(z_1, \ldots, z_5)$. Then 
$$ (J_{F,k}:m)=J_{F,k-1} \text{ for } k\geq 3. $$
\end{lem}
\begin{proof}
Note that $S^a V/J_F\otimes S^b V/J_F\to S^{a+b}V / J_F$ is non-degenerate in each factor if $a+b\leq 5$. Now, 
\begin{align*}
P\in (J_{F,k}:m^l)& \iff z^I P\in J_{F,k} \ \forall \abs{I}=l\\
& \iff z^I P=0 \text{ in } S^kV/J_F \ \forall \abs{I}=l\\
&\iff P=0 \text{ in } S^{k-l} V/J_F \text{ if } k\leq 5\\
& \iff P\in J_{F,{k-l}} \text{ if } k\leq 5
\end{align*}
\end{proof}

\begin{prop}\label{prop:Jf4:Q} The following equality holds
$$J_{C,2}^\perp = (J_{F,4}:Q).$$
\end{prop}

\begin{proof} 
Recall that $<,>$ denote the pairing between $S^k V $ and $S^k V^*$ for any $k$. Then for $P\in S^3 V$, 
\begin{align*}
<P,C>=0 &\iff PQ\in J_{F,5} && \text{ by Lemma \ref{lem:ZeroCubic} }\\
 & \iff P\in (J_{F,5}:Q). &&
\end{align*}
Thus $C$ is a basis for $(J_{F,5}:Q)^\perp \subset S^3 V^*$. If $z_1,\ldots, z_5$ is a basis for $V$, then 
$$<z_i P, R>=<P,\frac{\partial}{\partial z_i}R> \quad \text{ for all} P\in S^k V, R\in S^{k+1}V^*. $$
So for $P\in S^2 V$,
$$<z_i P, C>= <P, \frac{\partial C}{\partial z_i}>. $$
Thus, 
\begin{align*}
P\in J_{C,2}^\perp & \iff mP \in C^\perp && \\
&\iff mP\in (J_{F,5}:Q) &&\text{by Lemma \ref{lem:ZeroCubic}}\\
&\iff  mQP \in J_{F,5} &&\\
&\iff QP\in J_{F,4} &&\text{by Lemma \ref{lem:MaximalIdeal}}\\
&\iff P\in (J_{F,4}:Q). &&
\end{align*}
Thus $J_{C,2}^\perp = (J_{F,4}:Q)$.
\end{proof}

Let 
\begin{equation}\label{eq:mq}
S^2 V\overset{m_Q} \to S^4 V/J_{F,4}\cong \mathbb{C}^5
\end{equation}

\begin{cor} $C$ is non-degenerate if and only if $m_Q$ is surjective.
\end{cor}
\begin{proof}

We have that 
\begin{align*}
C \text{ is non degenerate} & \iff \frac{\partial C}{\partial z_1}, \ldots, \frac{\partial C}{\partial z_5} \text{ are linearly independent} && \text{by Lemma \ref{lem:PropCubics}} \\
& \iff \dim J_{C,2}=5 &&\\
&\iff \dim J_{C,2}^\perp=10 &&\\
&\iff \dim(J_{F,4}:Q)=10 &&\text{by Prop \ref{prop:Jf4:Q}}\\
&\iff S^2V \overset{m_Q}{\to} \frac{S^4 V}{J_{F,4}} \text{ has rank }5 &&\\
&\iff m_Q \text{ is surjective}.
\end{align*}
\end{proof}

The following proposition follows from Lemma \ref{lem:CubicAreSmooth}, which is well-known.
\begin{prop} $C$ is smooth if and only if $$\frac{S^k V^*}{J_C}=0,$$ for $k\geq 6$. 
\end{prop}
This allowed us to check via Macaulay2, see Appendix \ref{appendix:one}, that

\begin{thm}\label{thm:GenericSmooth} For generic $X$ and $Y$, the cubic $C$ is smooth.
\end{thm}
\subsubsection{Formula for C}
Recall from Proposition \ref{prop:Jf4:Q} that $\left( \frac{S^2V^*}{J_C} \right)^* \cong (J_{F,4}:Q)$, i.e. $J_C=(J_{F,4}:Q)^\perp$. Now, 
$$\left\langle \frac{\partial C}{\partial z_i^*}, P \right\rangle=0 \iff \left\langle C, z_i P \right\rangle=0 $$
so $C=\left( m\cdot (J_{F,4}:Q) \right)^\perp$. Note that $m\cdot (J_{F,4}:Q)\subseteq (J_{F,5}:Q)$, but it may not be equal. 

On the other hand, $J_{F,5}$ has codimension one in $S^5 V$, so 
$$S^3 V\overset{Q}\to \frac{S^5 V}{J_{F,5}} $$
has codimension one or zero in $S^3 V$, depending on whether $Q\in J_{F,2}$. We thus have proved the following Lemma. 

\begin{lem}\label{lem:spanC} If $Q\not \in J_{F,2}$, then $$span \  C = (J_{F,5}:Q)^\perp$$
\end{lem}

Let $V$ be a $\mathbb{C}$-vector space with $m=\dim V$. Take $F\in S^d V$ smooth. Then $S^{m(d-2)} V/J_F\cong \mathbb{C}$. We want to get information about the formula for the cubic form. Let us digress a moment and consider the following example. 

\begin{exmp} Let $m=2$ and fix $d$ arbitrary. Consider 
$$F=A_0 x^d +\ldots + A_d y^d $$ and the $(2d-3)\times (2d-3)$ matrix 
\begin{align*}
&\quad \quad {x^*}^{2d-4} \cdots {y^*}^{2d-4}\\
&\left( \quad \begin{array}{c}
x^{d-3} F_x \\
x^{d-4}y F_x \\
\vdots \\ 
x y^{d-4} F_x\\
y^{d-3} F_x\\
y^{d-3} F_y \\
\vdots \\ 
x^{d-4}y F_y \\
x^{d-4} F_y
\end{array} \quad \quad \quad   \right)
\end{align*}

This has degree $2(d-2)$ in the coefficients of $F$, and $$S^{2(d-2)}(S^d V)\to S^d V^*,$$
as
$$ S^{2(d-2)}(S^d V)\to (\det V)^{(d+1)(d-2)}\otimes S^2(d-2)V^*$$
we obtain 
\begin{prop} The following conditions are equivalent:
\begin{itemize}
\item The polynomial in $S^{2d-4} V^*$ we get is zero, 
\item There is a non-trivial solution to $P F_x +Q F_y=0$ with $\deg P=\deg Q= d-3$
\item $F_x,F_y$ have $2$ common roots, counting multiplicities, 
\item $F$ has $2$ singular points or one triple point.
\end{itemize}
\end{prop}
\begin{proof}
We will use Koszul complexes, see \cite{G84}.

Suppose that the singular locus $Z$ of $F$ is a smooth subvariety of $F$ of dimension $0$. 

We have 
$$0\to \Oo_{\mathbb{P}^1}(-2(d-1))[Z]\to \oplus_2 \Oo_{\Pp^1}(-(d-1))\to \Oo_{\Pp^1}\to \Oo_Z\to 0 $$
If $F_1=\mu_1\cdot s_Z, F_2=\mu_2\cdot s_Z$ near $Z$, where $s_Z$ defines $Z$, then the relations among $F_1, F_2$ are $\mu_2\cdot F_1-\mu_1 F_2=0$, and this accounts for the $[Z]$ on the left.

Put another way, the Koszul complex 

\begin{align*}
0\to &\Oo_{\mathbb{P}^1}(-2(d-1))&&\hspace{-2em}\to \oplus_2 \Oo_{\mathbb{P}^1}(-(d-1)) &&&\hspace{-90em}\to \Oo_{\Pp^1}\to 0 \\
\text{has homology} & && &&& \\
 &\quad 0   &&\Oo_Z   &&& \hspace{-2em}\hspace{-2em}\hspace{-2em}  \Oo_Z 
\end{align*}
Let 
\begin{align*}
R_k=\frac{S^k V}{J_{F,k}}, \quad e=2(d-2)\\ 
K_k=\ker(\oplus_2 S^k V\to S^{k+d-1}V)
\end{align*}
The hypercohomology spectral sequence has $E_2$ terms 
\[\begin{array}{lll}
R_{e-k}^* & K_{e-r-(d-1)}^* & 0 \\
0 & K_{k-(d-1)} & R_k
\end{array}
\]
It abuts to 

\[ 
\begin{array}{lll}
0 & 0 & 0\\
0 & \C^{\abs{Z}} & \C^{\abs{Z}}
\end{array}
\]

Thus we get an exact sequence 
$$0 \to K_{k-(d-1)}\to \C^{\abs{Z}}\overset{e_Z^*}\to R_{e-k}^*\to R_k\overset{e_Z}\to \C^{\abs{Z}}\to K_{e-k-(d-1)}^*\to 0 $$
where the map 
$$R_k\overset{e_Z}\to \C^{\abs{Z}} \text{ is } H^0(\Oo_{\Pp^1}(k))\to H^0(\Oo_Z(k)). $$
For $k=0$, $e_Z$ has rank one if $Z\not =\varnothing$. So 
$$0\to K_{-(d-1)}\to \C^{\abs{Z}}\to R_e^*\to R_0\to \C^{\abs{Z}}\to K_{e-(d-1)}^*\to 0 $$
Thus 
\begin{align*}
\abs{Z}=1 \Rightarrow K_{e-(d-1)}=0, \text{ i.e. } K_{d-3}=0\\
\abs{Z}\geq 2 \Rightarrow K_{d-3}\not = 0
\end{align*}

Taking $k=-1$, we have
$$0=R_{-1}\to \C^{\abs{Z}}\to K_{e+1-(d-1)}^*=K_{d-2}\to 0 $$
So $Z\not =\varnothing \iff K_{d-2}\not = 0$.
\end{proof}

\end{exmp}

Once the number of variables $m\geq 3$, we have the following generalization.

\begin{prop} $J_{F,m(d-2)}^\bot$ is given by the condition that $\sum P_iF_i=0$, $\deg P_i=d-3$, has a solution not generated by Koszul relation, and this is a polynomial of degree $$\left(\begin{array}{l} m(d-2)+m-1\\
m-1
\end{array} \right)-1$$ in the coefficients of $F$.
\end{prop}

\begin{exmp} Let us return to our main case when $m=5$ and $d=3$, so $m(d-2)=5$. There are ${5 \choose 2} 5 =50$ Koszul relations. Hence, the space of possible polynomials $P_i$ is $5 {7 \choose 3} -30=125$. Now, there are ${9\choose 4}=126$ monomials of degree $5$. Hence, we get a $126\times 126$ matrix whose first row is given by the monomials of degree $5$ in $z_1,\ldots, z_5$ and the others are given with by $z^J F_i$ with $\abs{J}=3$.
$$
\left( \begin{matrix}
 z_1^{*5} & \cdots & z_5^{*5} \\
  & z^J F_i & \\
  & & \ddots
\end{matrix}  \right)
$$

 The determinant of this matrix $M_F$ is the polynomial in $S^5 V^*$ that we search. Let $Q=\sum w_{ij}z_iz_j$ then $$P\in C^\perp \iff QP\in J_{F,5}\iff \left\langle QP, M_F \right\rangle=0 \iff \left\langle P,\sum w_{ij}\frac{\partial^2}{\partial z_i^* \partial z_j^*} M_F \right\rangle=0 $$

Hence $C=\sum w_{ij} \frac{\partial^2}{\partial z_i^* \partial z_j^*} M_F$. This is linear in the coefficients of $Q$, and of degree $125$ in the coefficients of $F$.
\end{exmp}

\subsection{Associated quintic}
\begin{prop}
If $C$ is smooth then $\frac{S^5 V^*}{J_C}\cong \mathbb{C}$. This isomorphism implies that $J_{C,5}^\perp=\{P\in S^5 V \mid d^3 P \subset (J_{F,4:Q})\}$ has dimension one.
\end{prop}
This assigns a quintic in $S^5 V$ to the pair $(X,Y)$. We will not study this quintic further here, but we expect it to have interesting geometry in the setting of the pairs $(X,Y)$ as above.

\subsection{Delta invariant for X fixed}
The purpose of this subsection is to prove Theorem \ref{thm:NonVanDeltaXfixed}. It will say that the infinitesimal invariant $\delta e_{Y/X}$ is non-zero when we held $X$ constant. In particular it will imply that the extension class $e_{Y/X}$ is non-zero. 

In the proof we will use the description of the geometric maps via different subquotients of $S^k V$. For this we will first prove some formulas that will be needed.
\subsubsection{Tangent space to deformations of Y } Recall that, following \cite{B04}, we denote $R=\Ima\{\Pic(X)\hookrightarrow \Pic(Y)\}$ and 
$$H^{1,1}(Y)_{\NEW}=\frac{H^{1,1}(Y)}{H^{1,1}(X)}\cong \frac{\Pic(Y)_\mathbb{C}}{\Pic(X)_\mathbb{C}}=R_\mathbb{C}^\perp. $$

\begin{prop}\label{prop:TDefY} We have an isomorphism 
$$H^{1,1}(Y)_{\NEW}\cong \frac{H^0(\Oo_Y(2))\oplus H^0(\Oo_Y(3))}{\{\sum A_i Q_i\oplus \sum A_i F_i\}},$$
where the $A_i$ are the same on both terms.
\end{prop}
\begin{proof}
We have a diagram 
\[\begin{tikzcd} & & 0 \ar[d] & & \\
& &  \Oo_Y \ar[d] & & \\
& & \oplus_5 \Oo_Y(1) \ar[d] & & \\
0 \ar[r] & T_Y \ar[r] & T_{\mathbb{P}^4}|_Y \ar[r] \ar[d] & N_{Y/\mathbb{P}^4} \ar[r] & 0 \\
& & 0  & & 
\end{tikzcd}
\]
and an isomorphism $N_{Y/\mathbb{P}^4}\cong \Oo_Y(2)\oplus \Oo_Y(3)$.
Now
$$\oplus_5 H^0(\Oo_Y(1))\to H^0(T_{\mathbb{P}^4}|_Y)\to H^1(\Oo_Y)=0 $$
$$0=\oplus_5 H^1(\Oo_Y(1))\to H^1(T_{\mathbb{P}^4}|_Y)\to H^2(\Oo_Y)\to \oplus H^2(\Oo_Y(1))=0 $$
and $H^2(\Oo_Y)\cong \C$. So
$$ \frac{H^0(\Oo_Y(2)\oplus H^0(\Oo_Y(3))}{\Ima(\oplus_5 H^0(\Oo_Y(1)))}\cong H^{1,1}(Y)_{\NEW} $$
i.e. 
$$H^{1,1}(Y)_{\NEW}\cong \frac{H^0(\Oo_Y(2))\oplus H^0(\Oo_Y(3))}{\{\sum A_i Q_i\oplus \sum A_i F_i\}}.$$
\end{proof}

Let us denote by $dQ\wedge dF$ the ideal in $\oplus_k H^0(\Oo_Y(k))$ generated by the $Q_iF_j-Q_jF_i$ with $Q_i$, respectively $F_i$, the $i$-th partial derivative of $Q$, respectively $F$. Similarly, denote by $\widetilde{dQ\wedge dF}$ the ideal in $\oplus_i H^0(\Oo_{\mathbb{P}^4}(k))$ generated by $dQ\wedge dF, Q$ and $F$.

\begin{prop} We have an isomorphism
$$ H^{1,1}(Y)_{\NEW}\cong \frac{S^3 V}{\widetilde{dQ\wedge dF}}.$$
\end{prop}

\subsubsection{Tangent space to deformation space of pairs} \begin{prop}\label{prop:TdefPairs} We have an isomorphism
$$ T\Def(X,Y)\cong \frac{S^3 V}{(dQ\wedge dF, F)} $$
\end{prop}
\begin{proof}
Recall that for $Z\subset \mathbb{P}^n$ smooth, there is a prolongation bundle $\Sigma_Z$ with extension class $c_1(H)\in H^1(\Omega_Z^1)$:
$$0\to \Oo_Z \to \Sigma_Z \to T_Z\to 0.$$

We have that $\Sigma_{\mathbb{P}^4}\cong \oplus_5 \Oo_{\mathbb{P}^4}(1)$ and 
$$0\to \Sigma_Z\to \Sigma_{\mathbb{P}^n}|_Z\to N_{Z/\Pp^n}\to 0 .$$

In the following diagram, we are searching for an explicit description of the bundle $E$.
\begin{equation}\label{eq:CDTDef(X,Y)}
\begin{tikzcd}
0 \ar[d] & 0\ar[d] & \\
E \ar[d] \ar[r] &  \Sigma_Y \ar[d] \ar[r] & 0\\
\Sigma_{\mathbb{P}^4}|_X \ar[r] \ar[d] & \Sigma_{\mathbb{P}^4}|_Y \ar[d] \ar[r] & 0 \\ \Oo_X(2)\oplus \Oo_X(3)\ar[d] \ar[r] & N_{Y/\mathbb{P}^4} \ar[d] \ar[r] & 0 \\
0 & 0 & 
\end{tikzcd}
\end{equation}
Now we have 
$$0\to \Sigma_X\to \Sigma_{\mathbb{P}^4}|_X\to \Oo_X(3)\to 0,$$
so
$$0 \to E\to \Sigma_X \to \Oo_X(2)\to 0.$$
Now $\Sigma_X$ acts as section of $\Oo_X(2)$ by differentiation, 
$$\Sigma_X\otimes \Oo_X(2)\to \Oo_X(2) $$
and we are acting on $Q$, so 
\[\begin{tikzcd}
 & & 0\ar[d] & 0\ar[d] & \\
  & & \Oo_X\ar[d]\ar[r,"\cong"] & Q\cdot \Oo_X \ar[d] & \\
0\ar[r] & E\ar[r]\ar[d] & \Sigma_X \ar[r, "{Q,dQ}"] \ar[d]  & \Oo_X(2)\ar[r] \ar[d]  & 0 \\
0\ar[r] & T_X(-\log Y)\ar[r] & T_X \ar[r,"dQ"] &\Oo_Y(2)\ar[r] &0 
\end{tikzcd}
\]
We have then that $E\cong T_X(-\log Y)$. 
Thus, since $H^1(\Sigma_{\mathbb{P}^4|_X})=0$ and using the diagram (\ref{eq:CDTDef(X,Y)}),we conclude that 
$$H^1(T_X(-\log Y))\cong \frac{H^0(\Oo_X(2))\oplus H^0(\Oo_X(3))}{\oplus_5 H^0(\Oo_X(1))}\cong \frac{S^2 V\oplus S^3 V/F}{\{\sum A_i Q_i\oplus \sum A_i F_i\}}. $$
\end{proof}
Using the prolongation bundle, we can give another proof of Proposition \ref{prop:TDefY}.

\subsubsection{Map from normal bundle of Y in X to its deformations}  The map 
$$\alpha:H^0(N_{Y/X})\to H^{1,1}(Y)_{\NEW} $$
is given by 
\begin{align*}
H^0(\Oo_Y(2))&\to \frac{H^0(\Oo_Y(2))\oplus H^0(\Oo_Y(3))}{\{\sum A_i Q_i \oplus \sum A_i F_i\}}. \\
P &\mapsto P\oplus 0
\end{align*}
Now $P=\sum A_i Q_i$ for some $A_i$, and $$P\mapsto -\sum A_i F_i \in \frac{S^3 V}{\widetilde{dQ\wedge dF}}. $$
\begin{lem}In particular
\begin{align*}
&\Ima(\alpha)\cong \frac{J_{F,3}}{J_{F,3}\cap (\widetilde{dQ\wedge dF})},\\
&\coker(\alpha)\cong \frac{S^3 V}{(J_{F,3},Q)}\cong \frac{H^{1,1}(Y)_{\NEW}}{\Ima H^0(N_{Y/X})}.
\end{align*}
\end{lem}
Now
$$P\in \ker(\alpha)\iff -\sum A_i F_i \in (dQ\wedge dF, F,Q),$$
we can choose the $A_i$ in this case so that $\sum A_i F_i\in (Q)$, i.e. $\sum A_i F_i\in  J_{F,3}\cap (Q)$. 
\begin{lem} Thus
$$\ker \alpha\cong J_{F,3}\cap Q \cong (J_{F,5}:Q) $$
\end{lem}
\begin{proof}
This follows by writing the diagram (\ref{eq:ComDiam}) for the cubic $3$-fold. By Propositions \ref{prop:TDefY} and \ref{prop:TdefPairs}, we have that 
\[\begin{tikzcd}
& & 0 \ar[d] & & \\
& & \frac{S^2 V}{Q}\ar[d] \ar[rd, "\sum A_i Q_i \mapsto \sum A_i F_i "] & & \\
0\ar[r] & V\ar[r] \ar[rd, "\cdot Q "] & \frac{S^3 V}{dQ\wedge dF, F} \ar[r] \ar[d] & \frac{S^3 V}{dQ\wedge dF, F,Q} \ar[r] & 0 \\
& & \frac{S^3 V}{J_F} \ar[d] & & \\
 & & 0 & &
\end{tikzcd}
\]
which can be completed as
\[\begin{tikzcd}
& & 0 \ar[d] & & \\
& & \frac{S^2 V}{Q}\ar[d] \ar[r, "="] & \frac{S^2 V}{Q} \ar[d] & \\
0\ar[r] & V\ar[r] \ar[d, " ="] & \frac{S^3 V}{dQ\wedge dF, F} \ar[r] \ar[d] & \frac{S^3 V}{dQ\wedge dF, F,Q} \ar[d] &  \\
& V\ar[r, "Q"]& \frac{S^3 V}{J_F} \ar[d] \ar[r] & \frac{S^3 V}{J_F,Q} & \\
 & & 0 & &
\end{tikzcd}
\]
\end{proof}
Using the first part of the Macaulay2 code we obtain. 
\begin{prop}\label{Prop:AlphaInjective} For generic $F,Q$, we have $(J_{F,3}:Q)=0$ and $\alpha$ is injective.
\end{prop}

\subsubsection{Proof of the Theorem} Let us explain what is $\delta e_{Y/X}$ when $X$ is held constant. 

We have two complexes
\begin{equation}\label{eq:XcteA}
\begin{tikzcd}
 H^{2,1}(X)^*\otimes H^{2,0}(Y) \arrow[r]
& \arrow[d, phantom, ""{coordinate, name=Z}] H^0(N_{Y/X})^*\otimes H^{2,1}(X)^*\otimes H^{1,1}(Y)_{\NEW} \arrow[d,
"\nabla" above,
rounded corners,
to path={ -- ([xshift=2ex]\tikztostart.east)
|- (Z) [near end]\tikztonodes
-| ([xshift=-2ex]\tikztotarget.west)
-- (\tikztotarget)}] \\
 & \wedge H^0(N_{Y/X})^*\otimes H^{2,1}(X)^*\otimes H^{0,2}(Y)
\end{tikzcd}
\end{equation}
\begin{equation}\label{eq:XcteB}
H^{1,2}(X)^*\otimes H^{1,1}(Y)_{\NEW}\to H^{0}(N_{Y/X})^*\otimes H^{1,2}(X)^*\otimes H^{0,2}(Y)\to 0
\end{equation}
The cohomology at the middle term of (\ref{eq:XcteB}) is $H^{1,2}(X)\otimes (\ker \alpha)^*$, i.e. a map
\[\begin{tikzcd} \ker\alpha \arrow[r] \arrow[d, "\cong"] & H^{2,1}(X) \arrow[d,"="]\\
(J_{F,3}:Q) \arrow[r, hook] & V
\end{tikzcd}
\]
By Proposition \ref{Prop:AlphaInjective}, we have that for generic $F$ and $Q$, this piece has no information. 

For (\ref{eq:XcteA}), we need an element of 
$$H^{2,1}(X)^*\otimes \left(\frac{H^0(N_{Y/X})^*\otimes H^{1,1}(X)_{\NEW}}{\C \alpha} \right). $$
Now this maps, losing some information, to a map
\[\begin{tikzcd}[contains/.style = {draw=none,"\cong" {description},sloped}]
H^{2,1}(X)\arrow[r] \ar[d,contains] & H^0(N_{Y/X})^* \ar[d,contains]  &  \mkern-18mu  \mkern-18mu \mkern-18mu \otimes  \left(\frac{H^{1,1}(X)_{\NEW}}{\Ima \alpha} \right)  \ar[d,contains]\\
V &  \left( \frac{S^2 V}{Q} \right) &   \frac{S^3 V}{(J_{F,3},Q)}
\end{tikzcd}\]

This map is induced by multiplication 
\[\begin{tikzcd} V\otimes \frac{S^2 V}{Q} \ar{r} \ar[rd] & \frac{S^3 V}{Q} \ar[d] \\
& \frac{S^3 V}{(J_{F,3},Q)}
\end{tikzcd} 
\]
We get $\delta e_{Y/X}\not = 0$ for $X$ fixed provided $ (J_{F,3},Q)\subsetneq S^3 V $, i.e. $V \overset{\alpha} \to \frac{S^3 V}{J_{F,3}}$ is not surjective. But $\dim (\frac{S^3 V}{J_{F,3}})=10, \dim V=5$. We have proved. 
\begin{thm}\label{thm:NonVanDeltaXfixed} For $X$ a smooth cubic threefold fixed and $Y\in \abs{-K_X}$ smooth, we have that $\delta e_{Y/X}\not = 0 $. More precisely, the part of $\delta e_{Y/X}$ in (\ref{eq:XcteB}) is non-zero.\end{thm}

\subsection{Generic Torelli Theorem for pairs}
\begin{thm} Let $(X,Y)$ be a generic pair in $\Fg$, then the MHS on $H^3(X\setminus Y)$ determines the pair $(X,Y)$.
\end{thm}
\begin{proof}
From the sequence (\ref{eq:DeltaInvFano}) we can recover $H^3(X)$, $H^2(Y)_{\NEW}$ and their polarizations from the MHS on $H^3(X,Y)$.

We therefore can recover the maps $T_{\Fg} \to T_{D_X}$ and $T_{\Fg}\to T_{D_Y}$ where $D_X$ and $D_Y$ are the period domains for $X$ and $Y$ respectively. It is enough to know infinitesimal Torelli for $X$ and for $Y$ to conclude that the kernels of the maps above are the kernels of $T_{\Fg}\to T_{\Def(X)}$ and $T_{\Fg}\to T_{\Kg}$. Recall that $H^{2,1}(X)=\ker(T_{\Fg}\to T_{D_Y})$ by (\ref{eq:BES4}).

The period map for $X$ gives $H^1(T_X)\to S^2H^{2,1}(X)^*$. To get the cubic form $C$, we need $$H^1(\Omega_X^2)\overset{s_Y}\to H^1(T_X)$$ with $s_Y\in H^0(-K_X)$.

We have 
\[
\begin{tikzcd}
 H^1(T_X(-Y))\cong \ker(T_{\Fg}\to T_{Kg}) \ar[rd] \ar[r, hook] & T_{\Fg} \ar[d] \\
 & T_{\Def(X)}
\end{tikzcd}
\]
We know that $H^1(T_X(-Y))\cong H^1(\Omega_X^2)$, with the isomorphism determined up to a constant.
So we have 

\[
\begin{tikzcd}
 H^1(\Omega_X^2) \ar[rd] \ar[r] & H^1(T_X) \ar[d] \\
 & S^2H^1(\Omega_X^2)^*
\end{tikzcd}
\]
is determined up to a constant, and hence: we can determine $C\in S^3 H^1(\Omega_X^2)^*$ up to a constant. Now $C^\perp \subset S^3 H^2(\Omega_X^1), C^\perp= (J_{F,5}:Q)\subset S^3 V$ by Lemma \ref{lem:spanC}.

Recall that for $X,Y$ generic, we have that $H^0(T_X|_Y)=0$ by Proposition \ref{Prop:AlphaInjective} and that $\ker(T_{\Fg}\to T_{D_X})\cong \frac{H^0(N_{Y/X})}{H^0(T_X|_Y)}$. Therefore, by using infinitesimal Torelli for $X$ we can recover $H^0(N_{Y/X})$.

 We now are going to use second order information:

For $\nu\in H^0(N_{Y/X})$, we want to look at 
$$\frac{\partial}{\partial \nu}(C^\perp)=\frac{\partial}{\partial \nu} ((J_{F,5}:Q)) $$

Now, because $X$ is staying fixed, we may regard $H^{2,1}(X)$ as a fixed space, and hence the isomorphism of $H^{2,1}(X)$ with $V$ is fixed.
$$\frac{\partial}{\partial \nu}((J_{F,5}:Q))\in \Hom((J_{F,5}:Q),\frac{S^3 V} {(J_{F,5}:Q)}) $$

\begin{lem} The following equality holds
$$\{P\in (J_{F,5}:Q)\mid \frac{\partial P}{\partial \nu}=0 \text{ for all }\nu \in H^0(N_{Y/X})\}=(J_{F,5}:m^2). $$
\end{lem}
\begin{proof}
 Suppose that $QP=\sum A_i F_i$, so $\frac{\partial Q}{\partial \nu} P=\sum_i \frac{\partial A_i}{\partial \nu} F_i$ this implies that  $\frac{\partial Q}{\partial \nu} P \in J_{F,5}$.
 Now we know 
 $$H^0(N_{Y/X})\cong H^0(\Oo_Y(2))\cong \frac{S^2 V}{Q}  $$
 and hence
 $$\frac{\partial P}{\partial \nu}=0 \text{ for all } \nu\in H^0(N_{Y/X}) \iff P\in (J_{F,5}:m^2) $$
\end{proof}

 Recall that 
$$S^2V\otimes \frac{S^3 V}{J_{F,3}}\to \frac{S^5 V}{J_{F,5}} $$
is non-degenerate in $\frac{S^3 V}{J_{F,3}}$, so
$$(J_{F,5}:m^2)=J_{F,3}, $$
this is, using this second order information, we can recover $J_{F,3}\subset S^3 V$.

 Now, using Mather-Yau's theorem we can recover $F$ from $J_F$.
 
Let us explain the dependence on $\GL(V)$ in order to differentiate as we have done:
$$S^3 H^{2,1}(X)\to H^0(-K_X)^* $$
is intrinsic to $X$, and $s_Y\in H^0(-K_X)$ is intrinsic to $(X,Y)$, hence $S^3 H^{2,1}(X)\overset{C}\to \C $ (up to a constant) is intrinsic to $(X,Y)$. It does not depend on how $X$ sits in $\mathbb{P}^4$ as a cubic. Likewise $H^0(N_{Y/X})\cong H^0(-K_X)/\mathbb{C}s_Y$ is intrinsic to $(X,Y)$, so 
$$H^0(N_{Y/X})\to \frac{S^3 H^{2,1}(X)^*}{\mathbb{C}C} $$
is intrinsic to $(X,Y)$.

Now \begin{align*}
 <\frac{\partial c}{\partial\nu},P>=0 \quad \forall \nu\in H^0(N_{Y/X}) \iff P\in \ker(S^3H^{2,1}(X)\to H^0(-K_X)^*)
\end{align*}

Now, choosing an embedding $X\subset \mathbb{P}^4$ as a cubic, 
$$\ker(S^3 H^{2,1}(X)\to H^0(-K_X)^*)\cong J_{F,3}$$ 
and we recover $J_{F,3}$ and hence $F$.

 We now know $F$, up to a constant, hence $X$. We also know $(J_{F,5}:Q)\subset S^3 V$

\begin{lem} We can recover $Q$ from $F$ and $(J_{F,5}:Q)$.
\end{lem}
\begin{proof}
We may assume that $Q$ is irreducible. Now if 
$$(J_{F,5}:Q)=(J_{F,5}:Q') $$
for $Q,Q'$ linearly independent, then

$$Q(J_{F,5}:Q)\subseteq (Q)\cap (Q')=(QQ')\subseteq S^5 V $$
But $(QQ')$ has dimension $5=\dim V$ in $S^5 V$, whereas $(Q)$ has dimension $35=\dim S^3 V$, and $(J_{F,5}:Q)$ has $\codim 1$ in $(Q)$ because $J_{F,5}$ has $\codim 1$ in $S^5 $. Hence $\dim (J_{F,5}:Q)= 34$.

This is a contradiction, so $Q'=c Q$.
\end{proof}
We thus have determined $F$ and $Q$ in $S^\ast H^{2,1}(X)$.

\end{proof}

\appendix
\section{Macaulay2 code}\label{appendix:one}
This is the code used in Macaulay2 for the proof of Theorem \ref{thm:GenericSmooth}.

\begin{verbatim}
KK=ZZ/31991
R=KK[a,b,c,d,e]
F=random(R^1, R^{-3})
Q=ideal(random(R^1, R^{-2}))
FJac = ideal(jacobian(F))
I2 =quotient(FJac,Q)
m=matrix{{a,b,c,d,e}}
m2=symmetricPower(2,m)
m3=symmetricPower(3,m)
m4=symmetricPower(4,m)
m5=symmetricPower(5,m)
m6=symmetricPower(6,m)
L = intersect(I2,ideal(m3))
Lgen=generators(L)
Ldiff= diff(transpose(m3),Lgen)
Lker= kernel(transpose(Ldiff))
C=m3*generators(Lker)
V=diff(transpose(m2), generators(I2))
W=generators(kernel(transpose(V)))
U=ideal(m2*W)
U5= generators(intersect(U, ideal(m5)))
Y = diff(transpose(m5),U5)
Quin= generators(kernel(transpose(Y)))
Quintic= m5*Quin
Z= intersect(I2,ideal(m6))
\end{verbatim}

We explain below the code.
\begin{enumerate}
\item \begin{verbatim}
KK=ZZ/31991
R=KK[a,b,c,d,e]
\end{verbatim}  \noindent We work in $\mathbb{Z}/31991$ so coefficients do not get too long. Now we work on $\mathbb{Z}/31991[a,b,c,d,e]$, polynomial ring in $5$ variables.
\item \begin{verbatim}
F=random(R^1, R^{-3})
Q=ideal(random(R^1, R^{-2}))
\end{verbatim} \noindent We get $F,Q$ random polynomials of degree $3$ and $2$.
\item \begin{verbatim}
FJac = ideal(jacobian(F))
\end{verbatim} \noindent $FJac$ is the ideal $J_F$.
\item \begin{verbatim}
I2 =quotient(FJac,Q)
\end{verbatim} \noindent  $I2$ is $(J_F:Q)$. It consists of $10$ graded generators.
\item \begin{verbatim}
m=matrix{{a,b,c,d,e}}
m2=symmetricPower(2,m)
m3=symmetricPower(3,m)
m4=symmetricPower(4,m)
m5=symmetricPower(5,m)
m6=symmetricPower(6,m)
\end{verbatim} \noindent This sets up matrices where entries are the monomials of degree $2, \ldots,6$.
\item  \begin{verbatim}
L = intersect(I2,ideal(m3))
Lgen=generators(L)
\end{verbatim} \noindent This computes the elements of $(J_F:Q)$ of degree $\geq 3$. As it happens, this is generated by cubics. There are $34$ of these in $S^3 V$, which has dimension $35$.
\item \begin{verbatim}
Ldiff= diff(transpose(m3),Lgen)
Lker= kernel(transpose(Ldiff))
C=m3*generators(Lker)
\end{verbatim} \noindent This is the main idea. We want to compute the unique (up to scalars) $C\in S^3 V^*$ that is orthogonal to the $34$ cubics found in the previous step. If $G_1,\ldots, G_{34}$ are these cubics, the first line computes the $35\times 34$ matrix whose rows are 
$$\frac{\partial^3 G_1}{\partial {z^{*}}^I}, \ldots, \frac{\partial^3 G_{34}}{\partial {z^{*}}^I}.$$
The entries are integers. The next line computes the kernel of the transpose - this kernel is of dimension one and is a vector of $35$ integers. There third line of this group converts these $35$ integers back into a cubic - this is $C$ (up to a constant.
\item \begin{verbatim}
V=diff(transpose(m2), generators(I2))
W=generators(kernel(transpose(V)))
U=ideal(m2*W)
\end{verbatim} \noindent This does something similar to the $10$ quadrics that generate $I_2$. The $5$-dimensional subspace of $S^2 V^*$ that is orthogonal to these $10$ quadrics is $\Span(\frac{\partial C}{\partial a}, \ldots, \frac{\partial C}{\partial e})$, i.e. $J_{C,2}$.
\item \begin{verbatim}
U5= generators(intersect(U, ideal(m5)))
Y = diff(transpose(m5),U5)
Quin= generators(kernel(transpose(Y)))
Quintic= m5*Quin
\end{verbatim} \noindent This computes the generators of $J_C\cap S^5 V^*$. These have $\codim 1$ in $S^* V^*$. Then it computes the one element of $S^5 V$ that is orthogonal to these, this is the quintinc.
\item \begin{verbatim}
Z= intersect(I2,ideal(m6)) 
\end{verbatim} This computes the generators of $J_C\cap S^6 V^*$. If $C$ is smooth, $S^6 V^*/ J_C=0$. If $C$ is singular, $J_C\cap S^7 V^*$ is contained in the set of sextics vanishing on the  singular locus of $C$. We check that $Z=J_C\cap S^6 V^*$ equals $S^6 V^*$, so $C$ is smooth.

\end{enumerate}

\bibliography{sample}
\bibliographystyle{amsalpha}
\end{document}